\documentclass[oneside,english]{amsart}
\usepackage[T1]{fontenc}
\usepackage[latin9]{inputenc}
\usepackage{amsthm}
\usepackage{amstext}
\usepackage{amssymb}
\usepackage{esint}
\usepackage{graphicx}
\usepackage{enumerate}
\usepackage{amscd}

\makeatletter
\newtheorem{theorem}{Theorem}[section]
\newtheorem{lemma}[theorem]{Lemma}

\numberwithin{figure}{section}
\theoremstyle{definition}

\theoremstyle{remark}
\newtheorem{remark}[theorem]{Remark}

\numberwithin{equation}{section}
\makeatother

\usepackage{color}

\usepackage{babel}

	\DeclareMathOperator{\dist}{dist}
	\DeclareMathOperator{\loc}{loc}
	\DeclareMathOperator{\ess}{ess}
	
	\DeclareMathOperator{\cp}{cap}
    \DeclareMathOperator{\adj}{adj}
	
\begin{document}

\title[On regularity of weighted Sobolev homeomorphisms]{On regularity of weighted Sobolev homeomorphisms}

\author{Valerii Pchelintsev and Alexander Ukhlov}

\begin{abstract}
We study the weak regularity of mappings inverse to weighted Sobolev homeomorphisms $\varphi:\Omega\to\widetilde{\Omega}$, where $\Omega$ and $\widetilde{\Omega}$ are domains in $\mathbb R^n$. Using the weak regularity of inverse mappings we obtain the composition duality property of composition operators on weighted Sobolev spaces.
\end{abstract}
\maketitle

\section{Introduction}

Composition operators on weighted Sobolev spaces were introduced in \cite{UV08} in the frameworks of the generalized Reshetnyak problem \cite{VGR}. On the base of weighted composition operators in \cite{GU09}  were obtained sharp weighted Sobolev inequalities in H\"older singular domains. Note that in the geometric theory of composition operators on Sobolev spaces the significant role have properties of inverse mappings \cite{U93,U04,VU98}. It arises to the invertibility properties of quasiconformal mappings \cite{V71} and have applications in the non-linear elasticity theory \cite{B81,HK04,Sv88}.

In the present article we consider a weak regularity of mappings $\varphi^{-1}$ which are inverse to weighted Sobolev homeomorphisms $\varphi:\Omega\to\widetilde{\Omega}$, where $\Omega$ and $\widetilde{\Omega}$ are domains (open and connected sets) in $\mathbb R^n$, generate by the composition rule $\varphi^{\ast}(f)=f\circ\varphi$ a bounded composition operators
$$
\varphi^{\ast}: L^1_p(\widetilde{\Omega})\to L^1_q(\Omega,w), \,\,n-1<q\leq p<\infty.
$$
The first main result of the article states: {\it Let a homeomorphism $\varphi$ belong to the weighted Sobolev space $L^{1}_{q}(\Omega,w;\widetilde{\Omega})$, $n-1<q<\infty$, $w^q\in A_q$ and $w^{(1-n)\widetilde{q}}\in A_{\widetilde{q}}$, $\widetilde{q}=q/(q-(n-1))$, possess the Luzin N-property and have finite distortion. Then the inverse mapping $\varphi^{-1}:\widetilde{\Omega}\to\Omega$ belongs to the Sobolev space $W^1_{1,\loc}(\widetilde{\Omega};\Omega)$.}

This weak regularity result generalizes the previous works on regularity of non-weighted Sobolev homeomorphisms \cite{CHM,GU10,HK06,HKM06,U04,V12,Z69}. Remark that in the case $n=2$ we have $\widetilde{q}=q/(q-(n-1))=q/(q-1)=q'$ and so the condition $w^{(1-n)\widetilde{q}}\in A_{\widetilde{q}}$ can be missed.

By using this weak regularity result we obtain the weighted composition duality property:
{\it Let a homeomorphism $\varphi:\Omega\to\widetilde{\Omega}$ have the Luzin $N$-property and generate a bounded composition operator
$$
\varphi^{\ast}: L^1_p(\widetilde{\Omega})\to L^1_q(\Omega,w), \,\,n-1<q\leq p<\infty,
$$
where $w^q \in A_q$ and  $w^{(1-n)\widetilde{q}}\in A_{\widetilde{q}}$. Then the inverse mapping $\varphi^{-1}:\widetilde{\Omega}\to \Omega$ generates composition operator
$$
\left(\varphi^{-1}\right)^{\ast}: L^1_{\widetilde{q}}(\Omega,v)\to L^1_{\widetilde{p}}(\widetilde{\Omega}), \,\,(n-1)/q+1/\widetilde{q}=1,(n-1)/p+1/\widetilde{p}=1,
$$
where $v(x)=w(x)^{1-n}$ such that $v^{\tilde{q}} \in A_{\tilde{q}}$.
}

Let us remind that in the geometric analysis of PDE composition operators on Sobolev spaces arise in \cite{M69}. On the base of composition operators in \cite{GGu,GS82} was suggested the approach to Sobolev extension operators and to Sobolev embedding operators. The necessary and sufficient conditions on mappings generate bounded composition operators on Sobolev spaces were obtained in \cite{U93}. In subsequent works \cite{GGR95,VU98,VU02,VU04,VU05} the geometric theory of composition operators on Sobolev spaces was founded. The sufficient conditions on mappings generate bounded composition operators were reproved in \cite{K12} by using another technique. By using the theory of multipliers the composition operators on Sobolev spaces were considered in \cite{MSh}.

Using the geometric theory of composition operators on Sobolev spaces the problem of estimates of Neumann eigenvalues in non-convex domains was solved, see for example, \cite{GPU18,GU16,GU17}.

It is known \cite{VG75} that a mapping $\varphi:\Omega\to\widetilde{\Omega}$, where $\Omega,\widetilde{\Omega}$ are domains in $\mathbb R^n$, $n\geq 2$, generates a bounded composition operator $\varphi^{\ast}: L^1_n(\widetilde{\Omega})\to L^1_n(\Omega)$
if and only if $\varphi$ is a quasiconformal mapping. Let us remind that a mapping $\varphi:\Omega\to\widetilde{\Omega}$ is called quasiconformal if $\varphi\in W^1_{n,\loc}(\Omega;\widetilde{\Omega})$ and $|D\varphi(x)|^n\leq K_n |J(x,\varphi)|$ for almost all $x\in\Omega$. Because the inverse mapping $\varphi^{-1}$ is quasiconformal also \cite{V71} it follows
$$
\left(\varphi^{-1}\right)^{\ast}: L^1_{n}(\Omega)\to L^1_{n}(\widetilde{\Omega}).
$$

In the general case of Sobolev spaces $L^1_p(\widetilde{\Omega})$ and $L^1_q({\Omega})$,  $n-1<q\leq p<\infty$, the corresponding result was proved in \cite{U93}. It was proved that if a mapping $\varphi:\Omega\to\widetilde{\Omega}$ generates a bounded composition operator $\varphi^{\ast}: L^1_p(\widetilde{\Omega})\to L^1_q(\Omega)$ then the inverse mapping $\varphi^{-1}$ generates a bounded composition operator $\left(\varphi^{-1}\right)^{\ast}: L^1_{\widetilde{q}}(\Omega)\to L^1_{\widetilde{p}}(\widetilde{\Omega})$, $\widetilde{q}=q/(q-n+1)$, $\widetilde{p}=p/(p-n+1)$. In the present work we generalize this composition duality property in the case of weighted Sobolev spaces defined in domains of $\mathbb R^n$.

\section{Regularity of weighted Sobolev homeomorphisms}

\subsection{Weighted Sobolev spaces}

Let $E$ be a measurable subset of $\mathbb R^n$, $n \geq 2$, and $w: \mathbb R^n \to [0,+ \infty)$ be a locally integrable nonnegative function, i.e., a weight.
The weighted Lebesgue space $L_p(E,w)$ is defined as a Banach space of locally integrable functions
$f: E \to \mathbb R$ endowed with the following norm:
\begin{align*}
\|f\mid L_{p}(E,w)\|=\biggr(\int\limits _{E}|f(x)|^{p} w(x)^p\, dx\biggr)^{\frac{1}{p}},\,\, 1\leq p<\infty,&\\
\|f\mid L_{\infty}(E,w)\|=\ess\sup\limits _{E}|f(x)|w(x),\,\, p=\infty.&
\end{align*}

Let $\Omega$ be an open subset of $\mathbb R^n$, $n \geq 2$. The weighted Sobolev space $W^{1}_{p}(\Omega,w)$, $1\leq p\leq\infty$, \cite{HKM} is defined
as a normed space of locally integrable weakly differentiable functions
$f:\Omega\to\mathbb{R}$ endowed with the following norm:
\[
\|f\mid W^{1}_{p}(\Omega,w)\|=\|f\mid L_{p}(\Omega,w)\| + \|\nabla f\mid L_{p}(\Omega,w)\|.
\]
The seminormed weighted Sobolev space $L^{1}_{p}(\Omega,w)$, $1\leq p\leq\infty$, is defined
as a space of locally integrable weakly differentiable functions
$f:\Omega\to\mathbb{R}$ endowed with the following seminorm:
\[
\|f\mid L^{1}_{p}(\Omega,w)\|=\|\nabla f\mid L_{p}(\Omega,w)\|.
\]
In the non-weighted case $w=1$ we have the classical seminormed Sobolev space $L^{1}_{p}(\Omega)$.
The weighted Sobolev space $W^1_{p,\loc}(\Omega,w)$ is defined as follows: $f\in W^1_{p,\loc}(\Omega,w)$
if and only if $f\in W^1_p(U,w)$ for every open and bounded set $U\subset  \Omega$ such that
$\overline{U}  \subset \Omega$.

Note that smooth functions are dense in $L^{1}_{p}(\Omega)$, $1 \leq p< \infty$ (see, for example, \cite{Bur,M}). If $p=\infty$, we can only assert that for an arbitrary function $f\in L^{1}_{\infty}(\Omega)$  there exists a sequence of smooth functions $\{f_k\}$ which converges locally uniformly to $f$ and $\|f_k\mid L^{1}_{\infty}(\Omega)\| \rightarrow \|f\mid L^{1}_{\infty}(\Omega)\|$ (see, for example, \cite{Bur}).

Without additional restrictions the weighted Sobolev spaces $W^{1}_{p}(\Omega,w)$ are not necessarily Banach spaces (see, for example, \cite{Kufn}). However, if the weight $w : \mathbb R^n \to [0,+ \infty)$ is such that the function $w^p$ satisfy the Muckenhoupt $A_p$-condition $(w^p\in A_p)$, $1<p<\infty$:
\[
\sup_{B \subset \mathbb R^n} \left(\frac{1}{|B|}\int\limits_{B} w^p\right)^{\frac{1}{p}}
\left(\frac{1}{|B|}\int\limits_{B} w^{-p'}\right)^{\frac{1}{p'}}<+ \infty ,\,\,\frac{1}{p}+\frac{1}{p'}=1,
\]
then $W^{1}_{p}(\Omega,w)$ is a Banach space and smooth functions of the class $W^{1}_{p}(\Omega,w)$ are dense in $W^{1}_{p}(\Omega,w)$ (see, \cite{GU09,HKM}). In the rest of the paper, let $W^{1}_{p}(\Omega,w)$ ($L^{1}_{p}(\Omega,w)$) be Sobolev spaces with $w^p\in A_p$.

\begin{remark}
\label{AP}
By the definition it follows that $w^p\in A_p$ if and only if $w^{-p'}\in A_{p'}$, $1/p+1/p'=1$.
\end{remark}

We consider the Sobolev spaces as Banach spaces of equivalence classes of functions up to a set of a weighted  $p$-capacity zero \cite{HKM,K94,M}. Recall the definition of the weighted $p$-capacity \cite{HKM,K94,M}.
Suppose $\Omega$ is an open set in $\mathbb R^n$ and  $F\subset\Omega$ is a compact set. The $(p,w)$-capacity of $F$ with respect to $\Omega$ is defined by
\begin{equation*}
\cp_{p,w}(F;\Omega) =\inf\{\|\nabla f|L_p(\Omega,w)\|^p\},
\end{equation*}
where the infimum is taken over all functions $f\in C_0(\Omega)\cap L^1_p(\Omega)$ such that $f\geq 1$ on $F$.
If
$U \subset\Omega$
is an open set, we define
\begin{equation*}
\cp_{p,w}(U;\Omega)=\sup_F\{\cp_{p,w}
(F;\Omega)\,:\,F\subset U,\,\, F\,\,\text{is compact}\}.
\end{equation*}

In the case of an arbitrary set
$E\subset\Omega$
we define the inner $(p,w)$-capacity
\begin{equation*}
\underline{\cp}_{p,w}(E;\Omega)=\sup_F\{\cp_{p,w}(F;\Omega)\, :\,\,F\subset E\subset\Omega,\,\, F\,\,\text{is compact}\},
\end{equation*}
and the outer $(p,w)$-capacity
\begin{equation*}
\overline{\cp}_{p,w}(E;\Omega)=\inf_E\{\cp_{p,w}(U;\Omega)\, :\,\,E\subset U\subset\Omega,\,\, U\,\,\text{is open}\}.
\end{equation*}

If $\underline{\cp}_{p,w}(E;\Omega)=\overline{\cp}_{p,w}(E;\Omega)$ then the value
$$
\cp_{p.w}(E;\Omega)=\underline{\cp}_{p,w}(E;\Omega)=\overline{\cp}_{p,w}(E;\Omega)
$$
is called the $(p,w)$-capacity measure of the set $E\subset\Omega$ \cite{HKM,K94}.

Let us note that a function $f\in W^{1}_{p}(\Omega,w)$ can be redefined in a set of measure zero so that it is quasicontinuous, i.e. its restriction to the complement of a set of arbitrary small $(p,w)$-capacity is continuous. Moreover, the quasicontinuous representative possesses quasi-everywhere Lebesgue points with respect to either Lebesgue measure or the weighted measure $w$ \cite{K94}.

\subsection{Weighted Sobolev homeomorphisms}

Let $\Omega$ and $\widetilde{\Omega}$ be domains (open and connected sets) in the Euclidean space $\mathbb R^n$. Then a homeomorphism $\varphi : \Omega \to \widetilde{\Omega}$ belongs to the weighted Sobolev space $W^1_p(\Omega,w;\widetilde{\Omega})$ if its coordinate functions belong to the weighted Sobolev space $W^1_p(\Omega,w)$. In this case $\varphi\in W^1_{1,\loc}(\Omega;\widetilde{\Omega})$ (Sobolev homeomorphism) and so its formal Jacobi matrix $D \varphi(x)$ and its determinant
(Jacobian) $J(x, \varphi)$ are well defined at almost all points $x \in \Omega$. The norm $|D \varphi(x)|$ of the matrix $D \varphi(x)$ is the norm of the corresponding linear operator.

Let us remind the change of variable formula in the Lebesgue integral \cite{H93}. Let $\varphi: \Omega \to \widetilde{\Omega}$ be a Sobolev homeomorphism of the class $W^1_{1,\loc}(\Omega;\widetilde{\Omega})$. Then there exists a measurable set $S\subset \Omega$, $|S|=0$ such that  the mapping $\varphi:\Omega\setminus S \to \mathbb R^n$ has the Luzin $N$-property and the change of variable formula
\begin{equation}
\label{CVF}
\int\limits_E f\circ\varphi (x) |J(x,\varphi)|~dx=\int\limits_{\widetilde{E}\setminus \varphi(S)} f(y)~dy
\end{equation}
holds for every measurable set $E\subset \Omega$, $\widetilde{E}=\varphi(E)\subset \widetilde{\Omega}$, and every non-negative measurable function $f: \widetilde{\Omega}\to\mathbb R$.
The mapping $\varphi : \Omega \to \widetilde{\Omega}$ possesses the Luzin $N$-property if an image of any set of
measure zero has measure zero. Lipschitz mappings possess the Luzin $N$-property (see, for example, \cite{M,VGR}).

In the following theorem we give a characterization of weighted Sobolev homeomorphisms in the terms of compositions with Lipschitz functions.

\begin{theorem}
\label{loc-unif} Let $\varphi:\Omega\to\widetilde{\Omega}$ be a homeomorphism between two domains $\Omega, \widetilde{\Omega}\subset \mathbb R^n$. Then $\varphi$ generates a bounded composition operator
\[
\varphi^{\ast}:L^{1}_{\infty}(\widetilde{\Omega})\to L^{1}_{q}(\Omega,w),\,\,\,1 < q < \infty,
\]
if and only if $\varphi$ belongs to the weighed Sobolev space $L^{1}_{q}(\Omega,w;\widetilde{\Omega})$.
\end{theorem}

\begin{proof}
\textit{Necessity}. Suppose the homeomorphism $\varphi$ generates a bounded composition operator
\[
\varphi^{\ast}:L^{1}_{\infty}(\widetilde{\Omega})\to L^{1}_{q}(\Omega,w),\,\,\,1 < q < \infty.
\]
Then there exists a constant $K_q(\Omega)<\infty$ such that the inequality
\begin{equation}
\label{eq1}
\|f\circ\varphi\mid L^{1}_{q}(\Omega,w)\| \leq K_q(\Omega) \|f\mid L^{1}_{\infty}(\widetilde{\Omega})\|
\end{equation}
holds for any function $f\in L^{1}_{\infty}(\widetilde{\Omega})$.

Now, substituting in inequality (\ref{eq1}) the test coordinate functions $f_i(y)=y_i\in L^{1}_{\infty}(\widetilde{\Omega})$, $i=1,\ldots,n$, we obtain
\[
\|\varphi^{*}_i\mid L^{1}_{q}(\Omega,w)\|= \|y_i\circ\varphi\mid L^{1}_{q}(\Omega,w)\|\leq K_q(\Omega) \|y_i\mid L^{1}_{\infty}(\widetilde{\Omega})\|= K_q(\Omega), \,\,i=1,...,n,
\]
because $\|y_i\mid L^{1}_{\infty}(\widetilde{\Omega})\|= 1$, $i=1,...,n$.
Hence we see that $\varphi$ belongs to $L^{1}_{q}(\Omega,w;\widetilde{\Omega})$.

\textit{Sufficiency}.
Since $\varphi$ belongs to the weighed Sobolev space $L^{1}_{q}(\Omega,w;\widetilde{\Omega})$ we denote by $K_q(\Omega):=\|\varphi\mid L^{1}_{q}(\Omega,w;\widetilde{\Omega})\|$.
Let $f \in L^{1}_{\infty}(\widetilde{\Omega}) \cap C^{\infty}(\widetilde{\Omega})$ be a smooth function of the class $L^{1}_{\infty}(\widetilde{\Omega})$. Then because $\varphi\in L^{1}_{q}(\Omega,w;\widetilde{\Omega})$  the composition $f\circ\varphi\in W^{1}_{1,\loc}(\Omega,\widetilde{\Omega})$ and we have
\begin{multline}
\label{eq2}
\|\varphi^{*}(f)\mid L^{1}_{q}(\Omega,w)\|
= \left(\int\limits_{\Omega}|\nabla (f \circ \varphi(x))|^q w^q(x)dx\right)^\frac{1}{q}\\
\leq \left(\int\limits_{\Omega} |D \varphi(x)|^q |\nabla f|^q(\varphi(x)) w^q(x)dx\right)^\frac{1}{q}
\leq \sup\limits_{x\in\Omega}|\nabla f|(\varphi(x))\left(\int\limits_{\Omega} |D \varphi(x)|^q w^q(x)dx\right)^\frac{1}{q}
\\
=\sup\limits_{y\in\widetilde{\Omega}}|\nabla f(y)|\left(\int\limits_{\Omega} |D \varphi(x)|^q w^q(x)dx\right)^\frac{1}{q}=
K_q(\Omega)\cdot
\|f\mid L^{1}_{\infty}(\widetilde{\Omega})\|.
\end{multline}

Now let $f\in L^{1}_{\infty}({\widetilde{\Omega}})$ be an arbitrary function. Then we consider a sequence of smooth functions $f_k\in L^{1}_{\infty}(\widetilde{\Omega})$
such that
\[
\lim_{k \to \infty} \|f_k\mid L^{1}_{\infty}(\widetilde{\Omega})\| =\|f\mid L^{1}_{\infty}(\widetilde{\Omega})\|
\]
and $\{f_k\}$ converges locally uniformly to $f$ in $\widetilde{\Omega}$ \cite{Bur}.

Denote by $g_k=f_k \circ \varphi$, $k=1,2,\ldots$. Then by the inequality (\ref{eq2}) the sequence $\{\varphi^{\ast}(f_k)\}$ converges locally uniformly to $\varphi^{\ast}(f)$ in $\Omega$ and is a bounded sequence in $L^{1}_{q}(\Omega,w)$. Since the Sobolev space $L^{1}_{q}(\Omega,w)$, $1<q<\infty$, is reflexive (see, for example, \cite{HKM,Y80}) then there exists a subsequence $\{f_{k_m}\}\in L^{1}_{q}(\Omega,w)$ which weakly converging to $f \in L^{1}_{q}(\Omega,w)$ and moreover
\[
\|\varphi^{\ast}(f)\mid L^{1}_{q}(\Omega,w)\| \leq \liminf_{m \to \infty} \|\varphi^{\ast}(f_{k_m})\mid L^{1}_{q}(\Omega,w)\|.
\]
Passing to limit as $m \to \infty$ in the inequality
\[
\|\varphi^{\ast}(f_{k_m})\mid L^{1}_{q}(\Omega,w)\| \leq K_q(\Omega) \|f_{k_m}\mid L^{1}_{\infty}(\widetilde{\Omega})\|,
\]
we have
\[
\|\varphi^{\ast}(f)\mid L^{1}_{q}(\Omega,w)\| \leq K_q(\Omega) \|f\mid L^{1}_{\infty}(\widetilde{\Omega})\|
\]
for any function $f\in L^{1}_{\infty}({\widetilde{\Omega}})$. The proof is complete.

\end{proof}

\subsection{Regularity of weighted Sobolev homeomorphisms}

On the base of Theorem~\ref{loc-unif} we prove the weak differentiability of mappings, which are inverse to weighted Sobolev homeomorphisms. This result generalizes a corresponding regularity results of non-weighted Sobolev homeomorphisms \cite{GU10}. Recall that a Sobolev homeomorphism $\varphi: \Omega \to \widetilde{\Omega}$ of the class $W^1_{1,\loc}(\Omega;\widetilde{\Omega})$ has finite distortion \cite{VGR} if
\[
D\varphi=0\,\,\, \text{a.e. on the set}\,\,\, Z=\{x \in \Omega:|J(x,\varphi)|=0\}.
\]

\begin{theorem}
\label{inverse}
Let a homeomorphism $\varphi:\Omega\to\widetilde{\Omega}$ between two domains $\Omega,\widetilde{\Omega}\subset\mathbb R^n$
belong to the weighted Sobolev space $L^{1}_{q}(\Omega,w;\widetilde{\Omega})$, $n-1<q<\infty$, $w(x)^{(1-n)\widetilde{q}}\in A_{\widetilde{q}}$, $\widetilde{q}=q/(q-n+1)$, possess the Luzin N-property and have finite distortion.
Then the inverse mapping $\varphi^{-1}:\widetilde{\Omega}\to\Omega$ belongs to the Sobolev space $W^1_{1,\loc}(\widetilde{\Omega};\Omega)$
and generates a bounded composition operator
$$
\left(\varphi^{-1}\right)^{\ast}: L^1_{\widetilde{q}}(\Omega,v)\to L^1_1(\widetilde{\Omega}),
$$
where the weight $v(x)=w(x)^{1-n}$ such that $v^{\widetilde{q}}\in A_{\widetilde{q}}$.
\end{theorem}

\begin{proof}
Let a function $f \in L^{1}_{\infty}(\widetilde{\Omega})$, then by Theorem~\ref{loc-unif} the composition $f \circ \varphi$ belongs to the Sobolev space $L^{1}_{q}(\Omega,w)$. Because the mapping $\varphi$ has finite distortion, we can define
$\adj D\varphi(x)w(x)=0$ for almost all $x\in Z=\{x\in\Omega:J(x,\varphi)=0\}$. Then
\[
|J(x,\varphi)| |\nabla f(\varphi(x))| \leq |\nabla f \circ \varphi(x)| |\adj D\varphi(x)|\,\, \text{for almost all}\,\, x\in\Omega,
\]
since
\[
\min_{|h|=1}|D\varphi(x) \cdot h|=\frac{1}{\max\limits_{|h|=1}|(D\varphi(x))^{-1} \cdot h|}
\]
and
\[
(D\varphi(x))^{-1}=\frac{\adj D\varphi(x)}{J(x,\varphi)},\,\,J(x,\varphi)\neq 0.
\]
Hence by using the change of variables formula and Luzin $N$-property we obtain
\begin{multline*}
\|f\mid L^{1}_{1}(\widetilde{\Omega})\|
= \int\limits_{\widetilde{\Omega}}|\nabla f(y)|~dy
=\int\limits_{\Omega}|\nabla f(\varphi(x))||J(x,\varphi)|~dx\\
\leq \int\limits_{\Omega} |\nabla f \circ \varphi(x)| |\adj D\varphi(x)|~dx
\leq \int\limits_{\Omega} |\nabla f \circ \varphi(x)||D\varphi(x)|^{n-1}~dx\\
= \int\limits_{\Omega} |\nabla f \circ \varphi(x)| w(x)^{1-n} |D\varphi(x)|^{n-1} w(x)^{n-1}~dx.
\end{multline*}
Now by using the H\"older inequality with exponents $\widetilde{q}=\frac{q}{q-(n-1)}$ and  $\frac{q}{n-1}$ we have
\begin{multline*}
\|f\mid L^{1}_{1}(\widetilde{\Omega})\|
\leq \left(\int\limits_{\Omega}|\nabla f \circ \varphi(x)|^{\widetilde{q}}w(x)^{(1-n)\widetilde{q}}dx\right)^{\frac{1}{\widetilde{q}}}
\left(\int\limits_{\Omega}|D\varphi(x)|^{q}w(x)^{q}dx\right)^{\frac{1}{q}}\\
= \|\varphi \mid L^{1}_{q}(\Omega,w;\widetilde{\Omega})\|\cdot \|\varphi^{*}f \mid L^{1}_{\widetilde{q}}(\Omega,v)\|,
\end{multline*}
where by the assumption $v(x)^{\widetilde{q}}=w(x)^{(1-n)\widetilde{q}}\in A_{\widetilde{q}}$.

Thus, we have the lower estimate of the norm of the composition operator
\[
\|f\mid L^{1}_{1}(\widetilde{\Omega})\| \leq \|\varphi \mid L^{1}_{q}(\Omega,w;\widetilde{\Omega})\|\cdot\|\varphi^{*}f \mid L^{1}_{\widetilde{q}}(\Omega,v)\|,
\]
where functions  $g(x)=\varphi^{*}(f)$ belong to $L^{1}_{q}(\Omega,w)$. Therefore, we can conclude that the inverse operator
\[
(\varphi^{-1})^{*}:L^{1}_{\widetilde{q}}(\Omega,v) \cap L^{1}_{q}(\Omega,w) \to L^{1}_{1}(\widetilde{\Omega})
\]
is a bounded operator.

Now we fix a cut-off function $\eta \in C_0^{\infty}(B(0,2))$ which equals 1 on $B(0,1)$. Substituting in the inequality
\[
\|(\varphi^{-1})^{*}g \mid L^{1}_{1}(\widetilde{\Omega})\| \leq K \|g \mid L^{1}_{\widetilde{q}}(\Omega,v)\|
\]
the test functions
$$
g_i=(x-x_0)_i \eta(\frac{x-x_0}{r}),\,\, i=1,\ldots, n,
$$
where $x_0\in \Omega$, $r< \dist(x_0,\partial \Omega)$ and $(x-x_0)_i$ means the $i$th coordinate of the vector $x-x_0$, we obtain
that the inverse mapping $\varphi^{-1}:\widetilde{\Omega}\to\Omega$ belongs to the Sobolev space
$W^{1}_{1,\loc}(\widetilde{\Omega};\Omega)$.

The mapping $\varphi^{-1}:\widetilde{\Omega}\to\Omega$ of the class $W^{1}_{1,\loc}(\widetilde{\Omega};\Omega)$ generates a bounded composition operator
$$
\left(\varphi^{-1}\right)^{\ast}: L^1_{\widetilde{q}}(\Omega,v)\to L^1_1(\widetilde{\Omega})
$$
if $\varphi^{-1}$ is a mappings of finite distortion and \cite{GU09,UV08}
$$
\int\limits_{\widetilde{\Omega}}\left(\frac{|D\varphi^{-1}(y)|^{\widetilde{q}}}{|J(y,\varphi^{-1})| v(\varphi^{-1}(y))^{\widetilde{q}}}\right)^{\frac{1}{\widetilde{q}-1}}dy<\infty.
$$

Since $\varphi$ has the Luzin $N$-property then $|J(y,\varphi^{-1})| v(\varphi^{-1}(y))^{\widetilde{q}}>0$ for almost all $y\in \widetilde{\Omega}$ and so $\varphi^{-1}$ is a mapping of finite distortion.
Now by using the change of variable formula~\eqref{CVF} and the Hadamard type inequality \cite{U93}
\[
|D\varphi^{-1}(y)| \leq \frac{|D \varphi(x)|^{n-1}}{|J(x,\varphi)|}
\]
which holds for almost all points $y=\varphi(x)\in \widetilde{\Omega}$, we obtain
\begin{multline*}
\int\limits_{\widetilde{\Omega}}\left(\frac{|D\varphi^{-1}(y)|^{\widetilde{q}}}{|J(y,\varphi^{-1})| v(\varphi^{-1}(y))^{\widetilde{q}}}\right)^{\frac{1}{\widetilde{q}-1}}~dy \\
\leq
\int\limits_{\widetilde{\Omega}}\left(
\frac{|D \varphi(\varphi^{-1}(y))|^{(n-1)\widetilde{q}}}{|J(\varphi^{-1}(y),\varphi)|^{\widetilde{q}}}
\frac{1}{|J(y,\varphi^{-1})| v(\varphi^{-1}(y))^{\widetilde{q}}}\right)^{\frac{1}{\widetilde{q}-1}}dy
\\
=\int\limits_{\widetilde{\Omega}} \frac{|D \varphi(\varphi^{-1}(y))|^q |J(y,\varphi^{-1})|~dy}{v(\varphi^{-1}(y))^{\frac{q}{n-1}}}
\leq \int\limits_{\Omega}|D\varphi(x)|^{q}w(x)^{q}dx <+\infty,
\end{multline*}
since $\varphi$ belongs to the weighted Sobolev space $L^{1}_{q}(\Omega,w;\widetilde{\Omega})$, $1<q<\infty$. Therefore
we can conclude that $\varphi^{-1}$ generates a bounded
composition operator from $L^{1}_{\widetilde{q}}(\Omega,v)$ to $L^{1}_{1}(\widetilde{\Omega})$, where $v=w^{1-n}$ and $(n-1)/q+1/\widetilde{q}=1$.
\end{proof}

The conditions $w^q\in A_q$ and $w^{(1-n)\widetilde{q}}\in A_{\widetilde{q}}$ are hold, for example, in the case of power weights $w=|x|^{\alpha}$. Let us reformulate Theorem~\ref{inverse} for power weights.

\begin{theorem}
\label{inversepower}
Let a homeomorphism $\varphi:\Omega\to\widetilde{\Omega}$ between two domains $\Omega,\widetilde{\Omega}\subset\mathbb R^n$
belong to the weighted Sobolev space $L^{1}_{q}(\Omega,|x|^{\alpha};\widetilde{\Omega})$, where $-{n}/{q}<\alpha<{n}/q'$ for $n-1<q<\infty$, possess the Luzin N-property and have finite distortion.
Then the inverse mapping $\varphi^{-1}:\widetilde{\Omega}\to\Omega$ belongs to the Sobolev space $W^1_{1,\loc}(\widetilde{\Omega};\Omega)$
and generates a bounded composition operator
$$
\left(\varphi^{-1}\right)^{\ast}: L^1_{\widetilde{q}}(\Omega,|x|^{\alpha(1-n)})\to L^1_1(\widetilde{\Omega}),\,\,\,|x|^{\widetilde{q}\alpha(1-n)}\in A_{\widetilde{q}},
$$
where $\widetilde{q}=q/(q-n+1)$.
\end{theorem}

\begin{proof}
By Theorem~\ref{inverse} it is sufficient to prove that $w^q=|x|^{\alpha q}\in A_q$ and $w^{(1-n)\widetilde{q}}=|x|^{\alpha (1-n)\widetilde{q}} \in A_{\widetilde{q}}$.
It is known that $|x|^{\beta}\in A_q$ if and only if $-n<\beta <n(q-1)$. Hence $w^q=|x|^{\alpha q}\in A_q$ if and only if
$$
-n<\alpha q <n(q-1),
$$
and $w^{(1-n)\widetilde{q}}=|x|^{\alpha (1-n)\widetilde{q}} \in A_{\widetilde{q}}$ only if
$$
-n<\alpha (1-n)\widetilde{q} <n(\widetilde{q}-1).
$$
So we obtain
$$
-\frac{n}{q}<\alpha<n\frac{q-1}{q}\,\,\,\text{and}\,\,\,-\frac{n}{q}=\frac{n(1-\widetilde{q})}{\widetilde{q}(n-1)}<\alpha<\frac{n}{\widetilde{q}(n-1)}.
$$
Since
$$
n\frac{q-1}{q}=\frac{n}{\widetilde{q}(n-1)}=\frac{n}{\widetilde{q}}\,\,\,\text{if}\,\,\,n=2,\,\,\,
n\frac{q-1}{q}>\frac{n}{\widetilde{q}(n-1)}\,\,\,\text{if}\,\,\,n\geq 3,
$$
then power weights $w=|x|^{\alpha}$ are such that $w^q\in A_q$ and $w^{(1-n)\widetilde{q}}\in A_{\widetilde{q}}$ if and only if
$$
-\frac{n}{q}<\alpha<\frac{n}{\widetilde{q}(n-1)}.
$$
\end{proof}

In the case of the two-dimensional Euclidean space $\mathbb R^2$ we have $\widetilde{q}=q'=q/(q-1)$ and $w^q\in A_q$ if and only if $w^{-q'}\in A_{q'}$. Hence in this case we can reformulate Theorem~\ref{inverse} as follows:

\begin{theorem}
Let a homeomorphism $\varphi:\Omega\to\widetilde{\Omega}$ between two domains $\Omega,\widetilde{\Omega}\subset\mathbb R^2$ belong to the weighted Sobolev space $L^{1}_{q}(\Omega,w;\widetilde{\Omega})$, $1<q<\infty$, possess the Luzin N-property and have finite distortion.
Then the inverse mapping $\varphi^{-1}:\widetilde{\Omega}\to\Omega$ belongs to the Sobolev space $W^1_{1,\loc}(\widetilde{\Omega};\Omega)$
and generates a bounded composition operator
$$
\left(\varphi^{-1}\right)^{\ast}: L^1_{q'}(\Omega,v)\to L^1_1(\widetilde{\Omega}),
$$
where the weight $v(x)=w(x)^{-1}$ such that $v^{q'}\in A_{q'}$.
\end{theorem}

\section{Composition operators on weighted Sobolev spaces}

\subsection{Composition operators and weighted quasiconformal mappings}

Let $\Omega$ and $\widetilde{\Omega}$ be domains in the Euclidean space $\mathbb R^n$, $n\geq 2$. Then
a homeomorphism $\varphi:\Omega\to \widetilde{\Omega}$ is called $w$-weighted $(p,q)$-quasiconformal, $1<q\leq p<\infty$,  if $\varphi$ belongs
to the Sobolev space $W^{1}_{1,\loc}(\Omega;\widetilde{\Omega})$, has finite distortion,
\[
K_{p,q}(\varphi;\Omega)= \left(\int\limits_{\Omega}\left(\frac{|D\varphi(x)|^p w(x)^p }{|J(x,\varphi)|}\right)^{\frac{q}{p-q}}dx\right)^{\frac{p-q}{pq}}<\infty, \,\,\text{if}\,\,1<q<p<\infty,
\]
and
\[
\hskip 2.2cm K_{p,p}(\varphi;\Omega)=\ess\sup\limits_{x\in\Omega}\left(\frac{|D\varphi(x)|^p w(x)^p }{|J(x,\varphi)|}\right)^{\frac{1}{p}}<\infty, \,\,\text{if}\,\,1<q=p<\infty.
\]
In the case $p=q$ such mappings are called as $w$-weighted $p$-quasiconformal mappings.

Let $\Omega$ and $\widetilde{\Omega}$ be domains in $\mathbb R^n$. Then a homeomorphism $\varphi:\Omega\to\widetilde{\Omega}$ induces a bounded composition
operator
\[
\varphi^{\ast}:L^1_p(\widetilde{\Omega})\to L^1_q(\Omega,w),\,\,\,1< q\leq p\leq\infty,
\]
by the composition rule $\varphi^{\ast}(f)=f\circ\varphi$, if the composition $\varphi^{\ast}(f)\in L^1_q(\Omega)$
is defined quasi-everywhere in $\Omega$ and there exists a constant $K_{p,q}(\varphi;\Omega)<\infty$ such that
\[
\|\varphi^{\ast}(f)\mid L^1_q(\Omega,w)\|\leq K_{p,q}(\varphi;\Omega)\|f\mid L^1_p(\widetilde{\Omega})\|
\]
for any function $f\in L^1_p(\widetilde{\Omega})$.

The mappings generate bounded composition operators on weighted Sobolev spaces were considered in \cite{UV08} in terms of the inverse distortion function. In the present work we give the characterization of the weighed composition operators in terms of the weighted $p$-distortion functions.

Recall the notion of a weighted variational $(p,w)$-capacity \cite{HKM}. Let $\Omega\subset \mathbb R^n$ be a domain, then a condenser in $\Omega$ is the pair $(F_0,F_1)$ of connected closed relatively to $\Omega$ sets $F_0,F_1\subset \Omega$. A continuous function $u\in L_p^1(\Omega,w)$ is called an admissible function for the condenser $(F_0,F_1)$,
if the set $F_i\cap \Omega$ is contained in some connected component of the set $\operatorname{Int}\{x\vert u(x)=i\}$,\ $i=0,1$. We call $(p,w)$-capacity of the condenser $(F_0,F_1)$ relatively to domain $\Omega$
the value
$$
{\cp}_{p,w}(F_0,F_1;\Omega)=\inf\|u\vert L_p^1(\Omega,w)\|^p,
$$
where the greatest lower bound is taken over all admissible for the condenser $(F_0,F_1)\subset\Omega$ functions. If $w=1$ we obtain the  variational $p$-capacity ${\cp}_{p}(F_0,F_1;\Omega)$.

\begin{lemma}
\label{lem:CapacityDescPP}
Let a homeomorphism $\varphi :\Omega\to \widetilde{\Omega}$,  where $\Omega,\widetilde{\Omega}$ are domains in $\mathbb R^n$,
generate a bounded composition operator
$$
\varphi^{\ast}: L^1_p(\widetilde{\Omega})\to L^1_q(\Omega,w),\,\,1<q\leq p<\infty.
$$
Then for every condenser $(F_0,F_1)\subset \Omega$ the inequality
$$
\cp_{q,w}^{1/q}(\varphi^{-1}(F_0),\varphi^{-1}(F_1);\Omega) \leq K_{p,q}(\varphi;\Omega)\cp_{p}^{1/p}(F_0,F_1;\widetilde{\Omega})
$$
holds.
\end{lemma}

\begin{proof}
Let $u$ be an admissible function for the condenser $(F_0,F_1)\subset\widetilde{\Omega}$,
then $u\circ\varphi$ be an admissible function for the condenser
$(\varphi^{-1}(F_0),\varphi^{-1}(F_1))\subset\Omega$.
Since $\varphi$
generates a bounded composition operator
$$
\varphi^{\ast}: L^1_p(\widetilde{\Omega})\to L^1_q(\Omega,w)
$$
then, by the definition of the capacity, we have
$$
{\cp}_{q,w}^{1/q}(\varphi^{-1}(F_0),\varphi^{-1}(F_1);\Omega)\leq\|\varphi^{\ast}(u)\vert L_q^1(\Omega,w)\|\leq K_{p,q}(\varphi;\Omega)\|u\,|\,L_p^1(\widetilde{\Omega})\|.
$$
Since $u$ is an arbitrary admissible function, we obtain
$$
\cp_{q,w}^{1/q}(\varphi^{-1}(F_0),\varphi^{-1}(F_1);\Omega)
\leq K_{p,q}(\varphi;\Omega)\cp_{p}^{1/p}(F_0,F_1;\widetilde{\Omega}).
$$
\end{proof}

By Lemma~\ref{lem:CapacityDescPP} it follows that an preimage of a set of $p$-capacity zero has $(q,w)$-capacity zero. Note, that if $\cp_{q,w}(E;\Omega)=0$, $E\ne\emptyset$, then the Hausdorff dimension of $E$ does not exceed $n-1$ \cite{K94}.

Now we give characterizations of composition operators on weighted Sobolev spaces in terms of $w$-weighted $(p,q)$-quasiconformal mappings.

\begin{theorem}
\label{CompThm} Let $\Omega$ and $\widetilde{\Omega}$ be domains in the Euclidean space $\mathbb R^n$. Then  a homeomorphism  $\varphi:\Omega\to{\widetilde{\Omega}}$ generates, by the composition rule $\varphi^{*}(f)=f \circ \varphi$, a bounded composition operator
\[
\varphi^{\ast}:L^{1}_{p}(\widetilde{\Omega})\to L^{1}_{p}(\Omega,w),\,\,\,1< p < \infty,
\]
if and only if $\varphi$ is a $w$-weighted $p$-quasiconformal mapping.
\end{theorem}

\begin{proof} \textit{Necessity.} Fix a cut function $\eta\in C_{0}^{\infty}(\mathbb R^n)$, which is equal to one on the ball $B(0,1)$ and is equal to zero outside of the ball $B(0,2)$. Consider the test functions
$$
f_j(y)=(y_j-y_{j0})\eta\left(\frac{y-y_0}{r}\right),\,\,\, j=1,...,n,
$$
where $y_j$ denotes the $j$-th coordinate. Then
\begin{multline}
|\nabla( f_j(y))|=\left|(\nabla(y_j-y_{j0}))\eta\left(\frac{y-y_0}{r}\right)+(y_j-y_{j0})\nabla\biggl(\eta\left(\frac{y-y_0}{r}\right)\biggr)\right|\\
=
\eta\left(\frac{y-y_0}{r}\right)+\frac{y_j-y_{j0}}{r}(\nabla\eta)\left(\frac{y-y_0}{r}\right)\leq 1+(\nabla\eta)\left(\frac{y-y_0}{r}\right)\leq C.
\nonumber
\end{multline}
Substituting in the inequality
$$
\|\varphi^{\ast}(f) \mid L^1_p(\Omega,w)\|\leq K_{p,p}(\varphi;\Omega) \|f\mid L^1_p(\widetilde{\Omega})\|
$$
the test functions
$$
f_j(y)=(y_j-y_{j0})\eta\left(\frac{y-y_0}{r}\right),\,\,\, j=1,...,n,
$$
we obtain that
\begin{equation}\label{EqFunPP}
\biggl(\int\limits_{\varphi^{-1}(B(y_0,r))}|D\varphi(x)|^pw(x)^p\,dx\biggr)^{\frac1/p}\leq CK_{p,p}(\varphi;\Omega)(r^n)^{1/p},
\end{equation}
where $C$ is a constant which depends on $n$ and $p$ only. Hence the homeomorphism $\varphi$ belongs to the weighed Sobolev space $W_{p,{\loc}}^1(\Omega,w)$ and as a consequence to $W^1_{1\loc}(\Omega)$.

Now we prove that $\varphi:\Omega\to\widetilde{\Omega}$ is a mapping of finite distortion. Let $Z=\{x\in \Omega : J(x,\varphi)=0\}$. We prove
$$
\int\limits_{Z}|D\varphi(x)|^pw(x)^p\,dx=0.
$$
Rewrite this integral as a sum of two integrals:
$$
\int\limits_{Z}|D\varphi(x)|^pw(x)^p\,dx
=\int\limits_{Z\setminus S}|D\varphi(x)|^pw(x)^p\,dx+
\int\limits_{Z \cap S}|D\varphi(x)|^pw(x)^p\,dx,
$$
where $S$ is the set from the change of variables formula (\ref{CVF}) on which the homeomorphism $\varphi$ has no the Luzin $N$-property.

Because $|S|=0$ then
$$
\int\limits_{Z\cap S}|D\varphi(x)|^pw(x)^p\,dx=0.
$$

Now we show that
$$
\int\limits_{Z\setminus S}|D\varphi(x)|^pw(x)^p\,dx=0.
$$
By the change of variable formula we have $|\varphi(Z\setminus S)|=0$. Fix an arbitrary number $\varepsilon>0$. Then there exists a family of balls $\{B(y_i,r_i)\}$ which covering the set $\varphi(Z\setminus S)$ such that the multiplicity of the covering $B(y_i,2r_i)$ is finite and $\sum\limits_i|B_i|<\varepsilon$.
Then by inequality (\ref{EqFunPP}) we obtain
$$
\int\limits_{Z\setminus S}|D\varphi(x)|^pw(x)^p\,dx\leq
\sum\limits_{i=1}^{\infty}\int\limits_{\varphi^{-1}(B(y_i,r_i))}|D\varphi(x)|^pw(x)^p\,dx
\leq C^p K^p_{p,p}(\varphi;\Omega) \sum\limits_{i=1}^{\infty}|B_i|.
$$
Since $\varepsilon$ is an arbitrary number then $\int\limits_{Z\setminus S}|D\varphi|^pw(x)^p\,dx=0$. Hence, $D\varphi=0$ a.~e. on $Z\setminus S$ and as a consequence $D\varphi=0$ a.e. on $Z$ and $\varphi$ is the mapping of finite distortion.

Now we apply to the left side of the inequality~(\ref{EqFunPP}) the change of variable formula (\ref{CVF}). Denote by $B:=B(y_0,r)$:
\begin{multline}
\biggl(\int\limits_{\varphi^{-1}(B)}|D\varphi(x)|^pw(x)^p\,dx\biggr)^{\frac{1}{p}}
=
\biggl(\int\limits_{\varphi^{-1}(B)\setminus S}|D\varphi(x)|^pw(x)^p\,dx\biggr)^{\frac{1}{p}}\\
=\biggl(\int\limits_{\varphi^{-1}(B)\setminus (S\cup Z)}
|D\varphi(x)|^pw(x)^p\,dx\biggr)^{\frac{1}{p}}
=
\biggl(\int\limits_{\varphi^{-1}(B)\setminus (S\cup Z)}
\frac{|D\varphi(x)|^pw(x)^p}{|J(x,\varphi)|}|J(x,\varphi)|\,dx\biggr)^{\frac{1}{p}}\\
=\biggl(\int\limits_{B(y_0,r)\setminus \varphi(S)}
\frac{|D\varphi(\varphi^{-1}(y))|^pw(\varphi^{-1}(y))^p}{|J(\varphi^{-1}(y),\varphi)|}\,dy\biggr)^{\frac{1}{p}}
\leq CK_{p,p}(\varphi;\Omega)|B|^{\frac{1}{p}}.
\nonumber
\end{multline}
Hence
$$
\biggl(\frac{1}{|B|}\int\limits_{B(y_0,r)\setminus \varphi(S)}
\frac{|D\varphi(\varphi^{-1}(y))|^pw(\varphi^{-1}(y))^p}{|J(\varphi^{-1}(y),\varphi)|}\,dy\biggr)^{\frac{1}{p}}
\leq CK_{p,p}(\varphi;\Omega).
$$

By using the Lebesgue theorem on differentiability of the integral by measure (see, for example, \cite{RR55,VU04}) we obtain
$$
\biggl(
\frac{|D\varphi(\varphi^{-1}(y))|^pw(\varphi^{-1}(y))^p}{|J(\varphi^{-1}(y),\varphi)|}\biggr)^{\frac{1}{p}}
\leq CK_{p,p}(\varphi;\Omega)\quad\text{ for almost all}\,\,\, y\in\widetilde{\Omega}\setminus\varphi(S).
$$
Because on the set $\Omega\setminus S$ the homeomorphism $\varphi$ has the Luzin $N$-property, finally we have
$$
\ess\sup\limits_{x\in\Omega}\left(\frac{|D\varphi(x)|^pw(x)^p}{|J(x,\varphi)|}\right)^{\frac{1}{p}}\leq CK_{p,p}(\varphi;\Omega)<\infty.
$$

\textit{Sufficiency.}
Let $f\in L_p^1(\widetilde{\Omega})$ be a smooth function. Then the composition $f\circ\varphi$
belongs to the class $W^1_{1,\loc}(\Omega)$ and the chain rule holds \cite{Z89}. So, we have
\begin{multline}
\|\varphi^*(f)\vert L_p^1(\Omega.w)\|=\biggl(\int\limits_\Omega|\nabla (f\circ\varphi(x))|^pw(x)^p\,dx\biggr)^{\frac{1}{p}}\\
\leq
\biggl(\int\limits_\Omega(|\nabla f(\varphi(x))|^pw(x)^p|D\varphi(x)|^p\,dx\biggr)^{\frac{1}{p}}\\
=\biggl(\int\limits_{\Omega\setminus Z}|\nabla f(\varphi(x))|^pw(x)^p|J(x,\varphi)|
\frac{|D\varphi(x)|^p}{|J(x,\varphi)|}\,dx\biggr)^{\frac{1}{p}}.
\nonumber
\end{multline}
Hence
\begin{multline}
\|\varphi^*(f)\vert L_p^1(\Omega,w)\|=\biggl(\int\limits_\Omega|\nabla (f\circ\varphi(x))|^pw(x)^p\,dx\biggr)^{\frac{1}{p}}\\
\leq
\ess\sup\limits_{x\in\Omega}\left(\frac{|D\varphi(x)|^pw(x)^p}{|J(x,\varphi)|}\right)^{\frac{1}{p}}\biggl(\int\limits_{\Omega\setminus Z}|\nabla f(\varphi(x))|^p|J(x,\varphi)|\,dx\biggr)^{\frac{1}{p}}\\
=K_{p,p}(\varphi;\Omega)\biggl(\int\limits_{\Omega\setminus (Z\cup S)}|\nabla f(\varphi(x))|^p|J(x,\varphi)|\,dx\biggr)^{\frac{1}{p}}\\
=K_{p,p}(\varphi;\Omega)\biggl(\int\limits_{\widetilde{\Omega}\setminus \varphi(S)}|\nabla f(y)|^p\,dy\biggr)^{\frac{1}{p}}
\leq K_{p,p}(\varphi;\Omega)\biggl(\int\limits_{\widetilde{\Omega}}|\nabla f(y)|^p\,dy\biggr)^{\frac{1}{p}}.
\nonumber
\end{multline}

Therefore we proved the required inequality
$$
\|\varphi^{\ast}(f) \mid L^1_p(\Omega,w)\|\leq K_{p,p}(\varphi;\Omega) \|f \mid L^1_p(\widetilde{\Omega})\|
$$
for every smooth function $f\in L^1_p(\widetilde{\Omega})$.

To extend the estimate onto all functions $f\in L^1_p(\widetilde{\Omega})$, $1< p<\infty$, consider a sequence of smooth functions $f_k\in L^1_p(\widetilde{\Omega})$, $k=1,2,...$, such that $f_k\to f$ in $L^1_p(\widetilde{\Omega})$ and $f_k\to f$ $p$-quasi-everywhere in $\widetilde{\Omega}$ as $k\to\infty$. Since the preimage $\varphi^{-1}(E)$ of the set $E\subset \widetilde{\Omega}$ of a $p$-capacity zero has the $(p,w)$-capacity zero, we have $\varphi^{\ast}(f_k)\to \varphi^{\ast}(f)$
$(p,w)$-quasi-everywhere in $\Omega$ as $k\to\infty$. Hence extension by continuity of the operator $\varphi^{\ast}$ $L^1_p(\widetilde{\Omega})\cap C^{\infty}(\widetilde{\Omega})$ to $L^1_p(\widetilde{\Omega})$ coincides with the composition operator $\varphi^{\ast}(f) = f\circ\varphi$.

\end{proof}

In the case $1<q<p<\infty$ the critical role have introduced in \cite{U93} set functions associated with composition operators. Recall that a nonnegative mapping $\Phi$ defined on open subsets of $\Omega$ is a monotone countably additive set function \cite{RR55,VU04} if

\noindent
1) $\Phi(U_1)\leq \Phi(U_2)$ if $U_1\subset U_2\subset\Omega$;

\noindent
2)  for any collection $U_i \subset U \subset \Omega$, $i=1,2,...$, of mutually disjoint open sets
$$
\sum_{i=1}^{\infty}\Phi(U_i) = \Phi\left(\bigcup_{i=1}^{\infty}U_i\right).
$$

The following lemma gives properties of monotone countably additive set functions defined on open subsets of $\Omega\subset \mathbb R^n$ \cite{RR55,VU04}.

\begin{lemma}
\label{lem:AddFun}
Let $\Phi$ be a monotone countably additive set function defined on open subsets of the domain $\Omega\subset \mathbb R^n$. Then

\noindent
(a) at almost all points $x\in \Omega$ there exists a finite derivative
$$
\lim\limits_{r\to 0}\frac{\Phi(B(x,r))}{|B(x,r)|}=\Phi'(x);
$$

\noindent
(b) $\Phi'(x)$ is a measurable function;

\noindent
(c) for every open set $U\subset \Omega$ the inequality
$$
\int\limits_U\Phi'(x)~dx\leq \Phi(U).
$$
\end{lemma}

Now we formulate the fundamental property of the composition operators of weighted Sobolev spaces \cite{UV08}.

\begin{theorem}\label{set-func}
Let the mapping $\varphi : \Omega \to \widetilde{\Omega}$ generate a bounded composition operator
\[
\varphi^{*}:L^{1}_{p}(\widetilde{\Omega})\to L^{1}_{q}(\Omega,w),\,\,\, 1<q<p<\infty.
\]
Then
\[
\Phi(\widetilde{A})=\sup\limits _{f\in L^1_p(\widetilde{\Omega}\cap C_0(\widetilde{A})} \left(\frac{\|\varphi^{*} \mid L^{1}_{q}(\Omega,w)\|}{\|f \mid L^1_p(\widetilde{A})\|}\right)^{\frac{pq}{p-q}},
\]
be a bounded monotone countably additive function defined by on open bounded subsets $\widetilde{A} \subset \widetilde{\Omega}$.
\end{theorem}

\begin{theorem}
\label{CompTh} Let $\Omega$ and $\widetilde{\Omega}$ be domains in the Euclidean space $\mathbb R^n$. Then  a homeomorphism  $\varphi:\Omega\to{\widetilde{\Omega}}$ generates, by the composition rule $\varphi^{*}(f)=f \circ \varphi$, a bounded composition operator
\[
\varphi^{\ast}:L^{1}_{p}(\widetilde{\Omega})\to L^{1}_{q}(\Omega,w),\,\,\,1<q< p < \infty,
\]
if and only if $\varphi$ is a $w$-weighted $(p,q)$-quasiconformal mapping.
\end{theorem}

\begin{proof}
{\it Necessity.} By Theorem \ref{set-func} the inequality
\begin{equation}
\label{eqphi}
\|\varphi^{\ast} f|{L}_{q}^{1}(\Omega,w)\|\leq \Phi(\widetilde{A})^{\frac{p-q}{pq}}
\|f|{L}_{p}^{1}(\widetilde{\Omega})\|,\,\,\,1< q<p<\infty,
\end{equation}
holds for any function $f\in {L}_{p}^{1}(\widetilde{\Omega})\cap C_0(\widetilde{A})$.

Fix a cut function
$\eta\in C_{0}^{\infty}(\mathbb R^n)$,
which is equal to one on the ball $B(0,1)$
and is equal to zero outside of the ball $B(0,2)$.
Substituting in the inequality (\ref{eqphi})
the test functions
$$
f_j(y)=(y_j-y_{j0})\eta\left(\frac{y-y_0}{r}\right),\,\,\, j=1,...,n,
$$
we see that
\begin{equation}\label{EqFun}
\biggl(\int\limits_{\varphi^{-1}(B(y_0,r))}|D\varphi(x)|^qw(x)^q\,dx\biggr)^{1/q}\leq
C\Phi(B(y_0,2r))^{\frac{p-q}{pq}}(r^n)^{1/p}.
\end{equation}
where $C$ is a constant which depends on $n$ and $p$ only. Hence the homeomorphism $\varphi$ belongs to the weighed Sobolev space $W_{q,{\loc}}^1(\Omega,w)$ and as a consequence to $W^1_{1\loc}(\Omega)$.

Now we prove that $\varphi$ is a mapping of finite distortion.
Let $Z=\{x\in \Omega : J(x,\varphi)=0\}$. We prove
$$
\int\limits_{Z}|D\varphi(x)|^qw(x)^q\,dx=0.
$$
Rewrite this integral as sum of two integrals:
$$
\int\limits_{Z}|D\varphi(x)|^qw(x)^q\,dx
=\int\limits_{Z\setminus S}|D\varphi(x)|^qw(x)^q\,dx+
\int\limits_{Z \cap S}|D\varphi(x)|^qw(x)^q\,dx,
$$
where $S$ is the set from the change of variables formula (\ref{CVF}) on which the homeomorphism $\varphi$ has no the Luzin $N$-property.

Because $|S|=0$ then
$$
\int\limits_{Z\cap S}|D\varphi(x)|^qw(x)^q\,dx=0.
$$
Now we prove that
$$
\int\limits_{Z\setminus S}|D\varphi(x)|^qw(x)^q\,dx=0.
$$
By the change of variable formula (\ref{CVF}) we have that $|\varphi(Z\setminus S)|=0$. Fix $\varepsilon>0$. Then there exists a family of balls $\{B(y_i,r_i)\}$ generates a covering of the set $\varphi(Z\setminus S)$ such that the multiplicity of the covering $B(y_i,2r_i)$ is finite and $\sum\limits_i|B_i|<\varepsilon$.
Then by inequality (\ref{EqFun}) we obtain
\begin{multline}
\int\limits_{Z\setminus S}|D\varphi(x)|^qw(x)^q\,dx\leq
\sum\limits_{i=1}^{\infty}\int\limits_{\varphi^{-1}(B(y_i,r_i))}|D\varphi(x)|^qw(x)^q\,dx\\
\leq C\sum\limits_{i=1}^{\infty}\Phi(B(y_i,2r_i))^{\frac{p-q}{p}}(r_i^n)^{q/p}\\
\leq C\sum\limits_{i=1}^{\infty}\Phi(B(y_i,2r_i))^{\frac{p-q}{p}}
(\sum\limits_{i=1}^{\infty}r_i^n)^{q/p}.
\nonumber
\end{multline}
Since $\varepsilon$ is an arbitrary number then $\int\limits_{Z\setminus S}|D\varphi(x)|^qw(x)^q\,dx=0$. Hence $D\varphi=0$ a.~e. on $Z\setminus S$ and $\varphi$ is the mapping of finite distortion.

Now rewrite inequality~(\ref{EqFun}) to the form
$$
\biggl(
\int\limits_{\varphi^{-1}(B(y_0,r))}|D\varphi(x)|^qw(x)^q\,dx
\biggr)^{\frac{p}{p-q}}
\leq C\frac{\Phi(B(y_0,2r))}{|B(y_0,2r)|}(r^n)^{\frac{p}{p-q}}
$$
and apply to the left side of this inequality the change of variable formula. Denote $B:=B(y_o,r)$:
\begin{multline}
\biggl(\int\limits_{\varphi^{-1}(B)}
|D\varphi(x)|^qw(x)^q\,dx\biggr)^{\frac{p}{p-q}}\\=
\biggl(\int\limits_{\varphi^{-1}(B)\setminus S}
|D\varphi(x)|^qw(x)^q\,dx\biggr)^{\frac{p}{p-q}}
=\biggl(\int\limits_{\varphi^{-1}(B)\setminus (S\cup Z)}
|D\varphi(x)|^qw(x)^q\,dx\biggr)^{\frac{p}{p-q}}\\
=
\biggl(\int\limits_{\varphi^{-1}(B)\setminus (S\cup Z)}
\frac{|D\varphi(x)|^qw(x)^q}{|J(x,\varphi)|}|J(x,\varphi)|\,dx\biggr)^{\frac{p}{p-q}}\\
=\biggl(\int\limits_{B(y_0,r)\setminus \varphi(S)}
\frac{|D\varphi(\varphi^{-1}(y))|^qw(\varphi^{-1}(y))^q}{|J(\varphi^{-1}(y),\varphi)|}\,dy\biggr)^{\frac{p}{p-q}}
\leq C\frac{\Phi(B(y_0,2r))}{|B(y_0,2r)|}(r^n)^{\frac{p}{p-q}}.
\nonumber
\end{multline}
Hence, we obtain the inequality
$$
\biggl(\frac{1}{r^n}\int\limits_{B(y_0,r)\setminus \varphi(S)}
\frac{|D\varphi(\varphi^{-1}(y))|^qw(\varphi^{-1}(y))^q}{|J(\varphi^{-1}(y),\varphi)|}\,dy\biggr)^{\frac{p}{p-q}}
\leq C\frac{\Phi(B(y_0,2r))}{|B(y_0,2r)|}.
$$

Using the Lebesgue theorem on differentiability of the integral and properties of the volume derivative of the countable-additive set functions \cite{RR55,VU04} we obtain
$$
\biggl(
\frac{|D\varphi(\varphi^{-1}(y))|^qw(\varphi^{-1}(y))^q}{|J(\varphi^{-1}(y),\varphi)|}\biggr)^{\frac{p}{p-q}}
\leq C\Phi'(y)\quad\text{ for almost all}\,\,\, y\in\widetilde{\Omega}\setminus\varphi(S).
$$
Integrating of the last inequality on an arbitrary open bounded subset $\widetilde{U}\subset\widetilde{\Omega}$ we obtain
\begin{multline*}
\int\limits_{\widetilde{U}\setminus\varphi(S)}\biggl(
\frac{|D\varphi(\varphi^{-1}(y))|^qw(\varphi^{-1}(y))^q}{|J(\varphi^{-1}(y),\varphi)|}\biggr)^{\frac{p}{p-q}}~dy
\leq C\int\limits_{\widetilde{U}\setminus\varphi(S)}\Phi'(y)~dy\\
\leq C\int\limits_{\widetilde{U}}\Phi'(y)~dy
\leq C \Phi(\widetilde{U})\leq C  K_{p,q}^{\frac{p-q}{pq}}(\varphi;\Omega).
\end{multline*}
Since the choice of $\widetilde{U}\subset\widetilde{\Omega}$ is arbitrary, we have
$$
\int\limits_{\widetilde{\Omega}\setminus\varphi(S)}\biggl(
\frac{|D\varphi(\varphi^{-1}(y))|^qw(\varphi^{-1}(y))^q}{|J(\varphi^{-1}(y),\varphi)|}\biggr)^{\frac{p}{p-q}}~dy
\leq C K_{p,q}^{\frac{p-q}{pq}}(\varphi;\Omega).
$$
Hence
\begin{multline}
\int\limits_\Omega\left(\frac{|D\varphi(x)|^pw(x)^p}
{|J(x,\varphi)|}\right)^{\frac{q}{p-q}}\,dx=
\int\limits_\Omega\frac{(|D\varphi(x)|w(x))^{\frac{pq}{p-q}}}
{|J(x,\varphi)|^{\frac{p}{p-q}}}|J(x,\varphi)|\,dx\\
=
\int\limits_{\widetilde{\Omega}\setminus\varphi(S)}\biggl(
\frac{|D\varphi(\varphi^{-1}(y))|^qw(\varphi^{-1}(y))^q}{|J(\varphi^{-1}(y),\varphi)|}\biggr)^{\frac{p}{p-q}}~dy
\leq C K_{p,q}^{\frac{p-q}{pq}}(\varphi;\Omega).
\nonumber
\end{multline}

{\it Sufficiency.}
Let $f\in L_p^1(\widetilde{\Omega})\cap C^\infty(\widetilde{\Omega})$,  then the composition $f\circ\varphi$
belongs to the class $W^1_{1,\loc}(\Omega)$ and the chain rule holds \cite{Z89}. So, we have
\begin{multline}
\|\varphi^* (f)\vert L_q^1(\Omega,w)\|=\biggl(\int\limits_\Omega|\nabla (f\circ\varphi)|^q w(x)^q\,dx\biggr)^{\frac{1}{q}}\\
\leq
\biggl(\int\limits_\Omega|\nabla f(\varphi(x))|^q|D\varphi(x)|^qw(x)^q\,dx\biggr)^{\frac{1}{q}}\\
=\biggl(\int\limits_{\Omega\setminus Z}|\nabla f(\varphi(x))|^q|J(x,\varphi)|^{\frac{q}{p}}
\frac{|D\varphi(x)|^q w(x)^q}{|J(x,\varphi)|^{\frac{q}{p}}}\,dx\biggr)^{\frac{1}{q}}.
\nonumber
\end{multline}
Using the H\"older inequality we obtain
\begin{multline}
\|\varphi^*(f)\vert L_q^1(\Omega,w)\|=\biggl(\int\limits_\Omega|\nabla (f\circ\varphi)|^q w(x)^q\,dx\biggr)^{\frac{1}{q}}\\
\leq
\biggl(\int\limits_{\Omega\setminus Z}|\nabla f (\varphi(x))|^p|J(x,\varphi)|\,dx\biggr)^{\frac{1}{p}}
\cdot\biggl(\int\limits_{\Omega\setminus Z}\left(\frac{|D\varphi(x)|^p w(x)^p}
{|J(x,\varphi)|}\right)^{\frac{q}{p-q}}\,dx\biggr)^{\frac{p-q}{pq}}\\
\leq
\biggl(\int\limits_{\Omega\setminus (Z\cup S)}|\nabla f (\varphi(x))|^p|J(x,\varphi)|\,dx\biggr)^{\frac{1}{p}}
\cdot\biggl(\int\limits_{\Omega}\left(\frac{|D\varphi(x)|^p w(x)^p}
{|J(x,\varphi)|}\right)^{\frac{q}{p-q}}\,dx\biggr)^{\frac{p-q}{pq}}.
\nonumber
\end{multline}

Now the application of the change of variable formula gives the required inequality
$$
\|\varphi^{\ast}(f) \mid L^1_q(\Omega,w)\|\leq K_{p,q}(\varphi;\Omega) \|f \mid L^1_p(\widetilde{\Omega})\|
$$
for every smooth function $f\in L^1_p(\widetilde{\Omega})$.

To extend the estimate onto all functions $f\in L^1_p(\widetilde{\Omega})$, $1<q< p<\infty$, consider a sequence of smooth functions $f_k\in L^1_p(\widetilde{\Omega})$, $k=1,2,...$, such that $f_k\to f$ in $L^1_p(\widetilde{\Omega})$ and $f_k\to f$ $p$-quasi-everywhere in $\widetilde{\Omega}$ as $k\to\infty$. Since the preimage $\varphi^{-1}(S)$ of the set $S\subset \widetilde{\Omega}$ of $p$-capacity zero has the $(q,w)$-capacity zero, we have $\varphi^{\ast}(f_k)\to \varphi^{\ast}(f)$
$q$-quasi-everywhere in $\Omega$ as $k\to\infty$. This observation leads us to the following conclusion: Extension by continuity of the operator $\varphi^{\ast}$ $L^1_p(\widetilde{\Omega})\cap C^{\infty}(\widetilde{\Omega})$ to $L^1_p(\widetilde{\Omega})$ coincides with the composition operator  $\varphi^{\ast}(f) = f\circ\varphi$.
\end{proof}

\subsection{Weighted composition duality theorem}

Now we prove the weighted composition duality property. It is well known that mappings which are inverse to quasiconformal mappings are quasiconformal also (see, for example, \cite{V71,VGR}). The following theorem refines this property for the case of weighted $(p,q)$-quasiconformal mappings.

\begin{theorem}
\label{DualityTh} Let a homeomorphism $\varphi:\Omega\to{\widetilde{\Omega}}$, $\Omega, \widetilde{\Omega} \subset \mathbb R^n$ have the Luzin $N$-property and generate a bounded
composition operator
\[
\varphi^{\ast}:L^{1}_{p}(\widetilde{\Omega})\to L^{1}_{q}(\Omega,w),\,\,\,n-1 < q < p < \infty,
\]
where $w^q\in A_q$ and $w^{(1-n)\widetilde{q}}\in A_{\widetilde{q}}$. Then the inverse mapping $\varphi^{-1}: \widetilde{\Omega} \to \Omega$ generates a bounded
composition operator
\[
(\varphi^{-1})^{\ast}:L^{1}_{\tilde{q}}({\Omega, v})\to L^{1}_{\tilde{p}}(\widetilde{\Omega}),\,\,
\tilde{q}=\frac{q}{q-n+1},\tilde{p}=\frac{p}{p-n+1},
\]
where $v(x)=w(x)^{1-n}$ such that $v^{\tilde{q}} \in A_{\tilde{q}}$.
\end{theorem}

\begin{proof} Since $\varphi$ generates the composition operator $\varphi^{\ast}:L^{1}_{p}(\widetilde{\Omega})\to L^{1}_{q}(\Omega,w)$ then by Theorem~\ref{CompTh} the homeomorphism $\varphi\in W^1_{q,\loc}(\Omega,w;\widetilde{\Omega})$, $q>n-1$. Hence by Theorem~\ref{inverse} the inverse mapping  $\varphi^{-1}:\widetilde{\Omega}\to\Omega$ belongs to the Sobolev space $W^1_{1,\loc}(\widetilde{\Omega};\Omega)$.
Now we prove that $\varphi^{-1}$ generates a bounded
composition operator
\[
(\varphi^{-1})^{\ast}:L^{1}_{\tilde{q}}({\Omega, v})\to L^{1}_{\tilde{p}}(\widetilde{\Omega}).
\]
We consider two cases:  $q<p$ and $q=p$.

\noindent
{\bf The case q<p}. By \cite{GU09,UV08} it sufficient to prove that $\varphi^{-1}$ is a mapping of finite distortion and
$$
\int\limits_{\widetilde{\Omega}}\left(\frac{|D\varphi^{-1}(y)|^{\tilde{q}}}{|J(y,\varphi^{-1})| v(\varphi^{-1}(y))^{\tilde{q}}}\right)^{\frac{\tilde{p}}{\tilde{q}-\tilde{p}}}dy <\infty.
$$
Since $\varphi:\Omega\to{\widetilde{\Omega}}$ has the Luzin $N$-property, then $|J(y,\varphi^{-1})| v(\varphi^{-1}(y))^{\tilde{q}}>0$
for almost all $y\in \widetilde{\Omega}$ and so $\varphi^{-1}$ is a mapping of finite distortion. By Theorem \ref{CompTh}
\[
K_{p,q}(\varphi;\Omega)= \left[\int\limits_{\Omega}\left(\frac{|D\varphi(x)|^p w(x)^p }{|J(x,\varphi)|}\right)^{\frac{q}{p-q}}dx\right]^{\frac{p-q}{pq}}< \infty.
\]

Denote by $Z=\{x \in \Omega: |J(x,\varphi)|=0\}$. Then \cite{U93}
\begin{equation}\label{est}
|D\varphi^{-1}(y)| \leq \frac{|D \varphi(x)|^{n-1}}{|J(x,\varphi)|},
\end{equation}
for almost all $x \in \Omega \setminus  Z$ and for almost all $y=\varphi(x) \in \widetilde{\Omega}$ because by the change of variables formula  $|\varphi(Z)|=0$.
Hence, we obtain
\begin{multline*}
\int\limits_{\widetilde{\Omega}}\left(\frac{|D\varphi^{-1}(y)|^{\tilde{q}}}{|J(y,\varphi^{-1})| v(\varphi^{-1}(y))^{\tilde{q}}}\right)^{\frac{\tilde{p}}{\tilde{q}-\tilde{p}}}dy
=
\int\limits_{\widetilde{\Omega} \setminus \varphi(Z)}\left(\frac{|D\varphi^{-1}(y)|^{\tilde{q}}}{|J(y,\varphi^{-1})| v(\varphi^{-1}(y))^{\tilde{q}}}\right)^{\frac{\tilde{p}}{\tilde{q}-\tilde{p}}}dy \\
\leq
\int\limits_{\widetilde{\Omega} \setminus \varphi(Z)}\left( \left(\frac{|D \varphi(\varphi^{-1}(y))|^{n-1}}{|J(\varphi^{-1}(y), \varphi)|}\right)^{\tilde{q}} \frac{1}{|J(y,\varphi^{-1})| v(\varphi^{-1}(y))^{\tilde{q}}} \right)^{\frac{\tilde{p}}{\tilde{q}-\tilde{p}}}dy \\
=
\int\limits_{\widetilde{\Omega} \setminus \varphi(Z)}\frac{|D \varphi(\varphi^{-1}(y))|^{\frac{pq}{p-q}}}{|J(\varphi^{-1}(y), \varphi)|^{\frac{p}{p-q}}} \left(\frac{1}{v(\varphi^{-1}(y))} \right)^{\frac{pq}{(n-1)(p-q)}}dy \\
=
\int\limits_{\Omega \setminus Z}\frac{|D \varphi(x)|^{\frac{pq}{p-q}}}{|J(x,\varphi)|^{\frac{p}{p-q}}} \left(\frac{1}{v(x)} \right)^{\frac{pq}{(n-1)(p-q)}} |J(x,\varphi)|dx \\
\leq
\int\limits_{\Omega}\left(\frac{|D\varphi(x)|^p}{|J(x,\varphi)|v(x)^{p/(n-1)}}\right)^{\frac{q}{p-q}}dx.
\end{multline*}
Putting $v(x)^{-1}=w(x)^{n-1}$ we obtain
\begin{multline*}
\int\limits_{\widetilde{\Omega}}\left(\frac{|D\varphi^{-1}(y)|^{\tilde{q}}}{|J(y,\varphi^{-1})| v(\varphi^{-1}(y))^{\tilde{q}}}\right)^{\frac{\tilde{p}}{\tilde{q}-\tilde{p}}}dy \\
\leq
\int\limits_{\Omega}\left(\frac{|D\varphi(x)|^p w(x)^p}{|J(x,\varphi)|}\right)^{\frac{q}{p-q}}dx
=K_{p,q}^{\frac{pq}{p-q}}(\varphi;\Omega)< \infty.
\end{multline*}
Hence \cite{GU09,UV08} $\varphi^{-1}:\widetilde{\Omega} \to \Omega$ generates a bounded composition operator
\[
(\varphi^{-1})^{\ast}:L^{1}_{\tilde{q}}({\Omega, v})\to L^{1}_{\tilde{p}}(\widetilde{\Omega}),
\]
where $\tilde{q}=q/(q-n+1)$, $\tilde{p}=p/(p-n+1)$.

\noindent
{\bf The case p=q.} By \cite{GU09,UV08} it sufficient to prove that $\varphi^{-1}$ is a mapping of finite distortion and
$$
\ess\sup\limits_{y\in\widetilde{\Omega}}\frac{|D\varphi^{-1}(y)|^{\tilde{p}}}{|J(y,\varphi^{-1})|v(\varphi^{-1}(y))^{\tilde{p}}} <\infty.
$$
Since $\varphi:\Omega\to{\widetilde{\Omega}}$ has the Luzin $N$-property, then $|J(y,\varphi^{-1})| v(\varphi^{-1}(y))^{\tilde{p}}>0$
for almost all $y\in \widetilde{\Omega}$ and so $\varphi^{-1}$ is a mapping of finite distortion. By Theorem \ref{CompThm}
\[
K_{p,p}(\varphi;\Omega)=\ess\sup\limits_{x\in\Omega}\frac{|D\varphi(x)|^p w(x)^p }{|J(x,\varphi)|}<\infty.
\]

By the change of variable formula~\eqref{CVF}, inequality~\eqref{est} and using the equation $v(x)=w(x)^{1-n}$ we obtain
\begin{equation*}
\ess\sup\limits_{y\in\widetilde{\Omega}}\frac{|D\varphi^{-1}(y)|^{\tilde{p}}}{|J(y,\varphi^{-1})|v(\varphi^{-1}(y))^{\tilde{p}}} \\
\leq
\ess\sup\limits_{x\in\Omega}\frac{|D\varphi(x)|^p w(x)^p }{|J(x,\varphi)|} =K_{p,p}(\varphi;\Omega)< \infty.
\end{equation*}

So, \cite{GU09,UV08} $\varphi^{-1}:\widetilde{\Omega} \to \Omega$ generates a bounded composition operator
\[
(\varphi^{-1})^{\ast}:L^{1}_{\tilde{p}}({\Omega, v})\to L^{1}_{\tilde{p}}(\widetilde{\Omega}),
\]
where $\tilde{p}=p/(p-n+1)$ and $v(x)=w(x)^{1-n}$ such that $v^{\tilde{p}} \in A_{\tilde{p}}$.
\end{proof}

In the case of power weights $w(x)=|x|^{\alpha}$ we can reformulate Theorem~\ref{DualityTh} as follows:

\begin{theorem}
Let $-n/q< \alpha <n/\tilde{q}(n-1)$ and let a homeomorphism $\varphi:\Omega\to{\widetilde{\Omega}}$, $\Omega, \widetilde{\Omega} \subset \mathbb R^n$, have the Luzin $N$-property and generate a bounded
composition operator
\[
\varphi^{\ast}:L^{1}_{p}(\widetilde{\Omega})\to L^{1}_{q}(\Omega,|x|^{\alpha q}),\,\,\,n-1 < q < p < \infty.
\]
Then the inverse mapping $\varphi^{-1}: \widetilde{\Omega} \to \Omega$ generates a bounded
composition operator
\[
(\varphi^{-1})^{\ast}:L^{1}_{\tilde{q}}(\Omega, |x|^{\alpha (1-n) \tilde{q}})\to L^{1}_{\tilde{p}}(\widetilde{\Omega}),
\]
where $\tilde{q}=q/(q-n+1)$ and $\tilde{p}=p/(p-n+1)$.
\end{theorem}

In the case of the two-dimensional Euclidean space $\mathbb R^2$ we have $\widetilde{q}=q'=q/(q-1)$ and $w^q\in A_q$ if and only if $w^{-q'}\in A_{q'}$. So, Theorem~\ref{DualityTh} takes the form:

\begin{theorem}
Let a homeomorphism $\varphi:\Omega\to{\widetilde{\Omega}}$ of plane domains $\Omega, \widetilde{\Omega} \subset \mathbb R^2$ have the Luzin $N$-property and generate a bounded
composition operator
\[
\varphi^{\ast}:L^{1}_{p}(\widetilde{\Omega})\to L^{1}_{q}(\Omega,w),\,\,\,1 < q \leq p < \infty.
\]
Then the inverse mapping $\varphi^{-1}: \widetilde{\Omega} \to \Omega$ generates a bounded
composition operator
\[
(\varphi^{-1})^{\ast}:L^{1}_{q'}({\Omega, v})\to L^{1}_{p'}(\widetilde{\Omega}),\,\, 1/p+1/p'=1,\,\, 1/q+1/q'=1,
\]
where a weight $v(x)=w(x)^{-1}$ such that $v^{q'}\in A_{q'}$.
\end{theorem}

\vskip 0.2cm

\textbf{Acknowledgments.} The first author was supported by the Ministry of Science and Higher Education of Russia (agreement No. 075-02-2022-884).

\vskip 0.3cm

Regional Scientific and Educational Mathematical Center, Tomsk State University, 634050 Tomsk, Lenin Ave. 36, Russia
							
\emph{E-mail address:} \email{vpchelintsev@vtomske.ru}   \\
			
Department of Mathematics, Ben-Gurion University of the Negev, P.O.Box 653, Beer Sheva, 8410501, Israel
							
\emph{E-mail address:} \email{ukhlov@math.bgu.ac.il}

\end{document}